\documentclass{article}
\usepackage{amsmath, amsfonts, amsthm, graphicx, geometry}

\newcommand{\V}{\mathcal{V}}

\newcommand{\si}{\textup{si}}

\theoremstyle{cupthm}
\newtheorem{thm}{Theorem}[section]

\newtheorem{prob}[thm]{Problem}
\theoremstyle{cupdefn}

\theoremstyle{cuprem}

\newtheorem*{unthm}{Theorem}

\newenvironment{enum}{
\begin{enumerate}
  \setlength{\itemsep}{1pt}
  \setlength{\parskip}{0pt}
  \setlength{\parsep}{0pt}}{\end{enumerate}
}

	\title{Axiomatisability of the Class of Monolithic Groups in a Variety 
	of Nilpotent 
	Groups} 
\author{Joshua T. Grice}

\begin{document}

\maketitle
	
	\begin{abstract}
		The class of all subdirectly irreducible groups belonging to a variety 
		generated by a finite nilpotent group can be axiomatised by a finite 
		set of 
		elementary sentences. 
	\end{abstract}

\section{Introduction}
A group is \emph{subdirectly irreducible} or \emph{monolithic} if it has some 
minimal nontrivial 
normal subgroup, called its monolith. A \emph{variety} of groups is a class of 
groups 
closed with respect to homomorphic images, subgroups, and arbitrary direct 
products. According to a fundamental result of Garrett Birkhoff in his 1935 
paper
\cite{Birkhoff1935},  a variety is also a proper class of groups that is 
axiomatised by some set of equations. Given a group $\mathbf{G}$, the variety 
generated by $\mathbf{G}$, denoted $\V(\mathbf{G})$, is the smallest variety to 
which $\mathbf{G}$ belongs, or equivalently the intersection of all varieties 
containing $\mathbf{G}$. 

By another theorem of Birkhoff's in a 1944 paper \cite{Birkhoff1944}, two 
varieties are equal 
if they share the same subdirectly irreducible members. Thus, given a variety 
$\V$, the class of its subdirectly irreducible members, commonly denoted 
$\V_{\si}$, is of particular interest. Subdirect irreduciblity as a property is 
not preserved by the formation of direct products, and so $\V_{\si}$ is not a 
variety and cannot be determined by a finite set of equations. But it can still 
be axiomatised by a finite set of elementary sentences, which are more general 
first-order statements. 

The main result of this paper is as follows. 
\begin{thm}\label{main}
Let $\V$ be a variety generated by a finite nilpotent group $\mathbf{G}$. Then, the class of subdirectly irreducible groups belonging to $\V$ is axiomatisable by a finite set of elementary sentences.
\end{thm}

Finite axiomatisability of various structures has been studied at length by 
logicians, abstract algebraists and universal algebraists. In the realm of 
groups, Lyndon proved in 1952 \cite{Lyndon1952} that the variety generated by 
any nilpotent group is 
finitely axiomatisable. In 1965, Oates and Powell \cite{OatesPowell1964} proved 
that the variety generated by any  
finite group is finitely axiomatisable. In both of these cases, the axioms can 
be taken to be equations. 

A broader result than our main theorem was a hair's breadth from being proved 
by George F. McNulty and Wang Ju in 2000. A preprint was circulated in which it 
was claimed that the subdirectly 
irreducible members of a variety generated by any finite group was finitely 
axiomatisable. However, 
an error in 
the proof caused much of the work to be 
unusable. The current paper is an attempt to salvage some of that work, and has 
managed to prove the assertion with the added hypothesis of nilpotence. The 
author wishes to thank Dr. McNulty personally for his devoted guidance as a PhD 
advisor, and for this problem in particular. 

Baker and Wang proved in 2002 in \cite{BakerWang2002} that for certain kinds of 
varieties, $\V$ itself being finitely axiomatisable and $\V_{\si}$ being 
finitely 
axiomatisable are in fact equivalent. It is that result, and the aforementioned 
importance of subdirectly irreducible members of varieties, that motivates 
investigation into finite axiomatisability of subdirectly irreducible algebras.

\section{Preliminaries}

\subsection{Elementary Logic}
Our result dwells in first-order, or elementary logic, using the language of 
groups. The \emph{terms} of the language of groups are built up from variables 
representing elements of the group joined together with 
the multiplication and inverse operations of groups, and the named constant 
symbol 1 representing the identity. For instance, the 
conjugate $x y x^{-1}$ is a term in this language. The variables cannot be 
used to represent sequences or subgroups, only elements; hence the name 
elementary 
logic. 

An \emph{equation} is some statement of equality between two terms whose 
variables are understood to range over the whole group. For 
instance, the commutative law of Abelian groups can be expressed as an 
equation:
$$xy \approx yx$$
An \emph{elementary formula} is built up from equations in a systematic way 
with the help of
logical connectives $\vee, \wedge, \neg, \leftarrow,$ and $\leftrightarrow$ 
(conjunction, disjunction, negation, implication and biconditional, 
respectively), and the 
quantifiers $\exists$ and $\forall$. A formula may look something like:
$$\forall y \, (xy\approx yx) \wedge \neg(x\approx 1)$$
Note that in this formula, the variable $x$ appears but is not quantified. This 
makes $x$ a \emph{free variable,} and illustrates how formulas can be used to 
define sets of elements. If the above formula is named $\Phi$, for instance, 
the set defined by $\Phi(x)$ would be the set of all elements $x$ of a group 
that satisfy 
that formula. In this case, $\Phi(x)$ is the set of nontrivial elements of the 
group's center. 

If a formula has no free variables, it is called an \emph{elementary sentence}. 
A sentence in the language of groups is either true or false in a given group, 
whereas a formula depends on the value that the variables take. 
Sentences are useful for stating laws obeyed in a structure that cannot be 
expressed by equations alone. For example, the presence of an inverse for 
every element of a group: 
$$\forall x \,  \exists y \, (xy\approx 1)$$

Let $\mathcal{K}$ be a proper class of structures. If $\Sigma$ is a finite set 
of sentences so that a given structure $\mathbf{A}$ belongs to $\mathcal{K}$ if 
and only if it satisfies every sentence in $\Sigma$, we say that $\mathcal{K}$ 
is \emph{finitely axiomatisable}. When axiomatising varieties, we note that 
these sentences can all be taken to be equations. 

\subsection{Group Theory}
We define the \emph{normal closure} of a set $X$ of elements of a group 
$\mathbf{G}$ as the smallest normal subgroup containing $X$. We may also call 
it the normal subgroup generated by $X$. We denote this subgroup by $X^G$. If 
$X$ is a singleton set, say $X=\{a\}$, we call this subgroup a \emph{principal 
normal subgroup} and write $a^G$. Note that in a subdirectly irreducible group, 
the monolith cannot contain any nontrivial normal subgroup and is therefore 
always principal. 

Given two elements $a, b$ of a given group $\mathbf{G}$, their 
\emph{commutator} $[a,b]$ is the element $aba^{-1}b^{-1}$. The commutator 
operation can be extended to normal subgroups; the commutator of two normal 
subgroups  
$\mathbf{H}$
and $\mathbf{K}$ of $\mathbf{G}$ is defined as
$[\mathbf{H}, \mathbf{K}]=\{[h,k]: \, h\in  H, \,  k \in K\}$. 
The commutator of two normal subgroups is again a normal subgroup. Using the 
commutator 
operation, one may fabricate a \emph{lower central series} $G_0 \triangleright 
G_1 \triangleright G_2 \triangleright
\ldots$ where $G_0=\mathbf{G}$ and $G_i=[G, G_i]$. The group $\mathbf{G}$ is 
called \emph{nilpotent} of class $k$ if there is some $k$ for which 
$G_k=\{1\}$. 

An equivalent (and, for our purposes, more useful) definition of nilpotence
is the presence of an \emph{upper central series} $Z_0 \triangleleft Z_1 
\triangleleft Z_2 \triangleleft \ldots $ so that $\{1\}=Z_0$ and, for each $i$, 
$Z_{i+1}/Z_i=Z(\mathbf{G}/Z_i)$. The group $\mathbf{G}$ is nilpotent if there 
is some $k$ for which $Z_k=\mathbf{G}$. It is well-known that the length of the 
upper and lower 
central series coincides; a proof can be found in Dummit and Foote  
\cite{DummitFoote}. A third well-known characterisation of nilpotence exists; a 
group is 
nilpotent if and only if it is the direct product of its Sylow 
subgroups.

By Lyndon's work, the nilpotence class of a group can be captured with a 
finite 
set of equations \cite{Lyndon1952}. Therefore, if $\mathbf{G}$ is nilpotent of 
class $k$, any 
group $\mathbf{H} \in \V(\mathbf{G})$ is nilpotent of class at most $k$. 

Given a group $\mathbf{G}$, the normal subgroups of $\mathbf{G}$ form a 
lattice. If $K < H$ are normal subgroups of $\mathbf{G}$ and there exists no 
normal subgroup $N$ so that $H \leq N \leq K$, then $H/K$ is called a 
\emph{chief factor} of $\mathbf{G}$. According to Hanna Neumann's 1967 book 
\cite{Neumann}, the 
cardinality of chief factors in a variety generated by a finite group 
$\mathbf{G}$ is bounded above by $|\mathbf{G}|$. 

\section{Definable Principal Normal Subgroups}

Our proof of the finite axiomatisability of $\V_\si$ is contingent on a useful 
definition from Baker and Wang's 2002 paper \cite{BakerWang2002} that they 
called \emph{definable principal subcongruences.} Their formulation of this 
concept is from a perspective of general algebraic structures, so we will 
rework it here to focus purely through the spyglass of group theory. 

Let $\Phi(x,y)$ be an elementary formula. We will say that $\Phi$ is a 
\emph{normal closure formula} provided that for any group $\mathbf{H}$, if 
$\Phi(a,b)$ holds in $\mathbf{H}$, then $a$ belongs to the normal closure 
$b^\mathbf{H}$. For instance, the formula $$\exists z\, (x\approx zyz^{-1})$$ 
is 
a 
normal closure formula, since any conjugate of the element $b$ will belong to 
$b^\mathbf{H}$. Normal closure formulas are useful for capturing the principal 
normal subgroups of a given group or class of groups in a way that is 
compatible with 
first-order logic, which in turn can help to axiomatise the groups themselves. 
A class of groups might be highly compatible with such a 
capturing; for some classes of groups, there might be one normal closure 
formula $\Phi(x,y)$ that can define every principal normal subgroup of every 
group in the 
class. Baker and Wang's definition is not quite so strong as that, but is in 
many ways the next best thing. 

We will say that a class $\mathcal{K}$ of groups has \emph{definable principal 
normal subgroups} if and only if there are normal closure
formulas $\Phi(x,y)$ and $\Psi(x,y)$ so that for every $\mathbf{H} \in 
\mathcal{K}$ and every nonidentity $b \in H$, there exists a nonidentity $a \in 
H$ so that 
\begin{enum}
	\item $\mathbf{H} \models \Psi(a,b)$ and
	\item $\Phi(x,a)$ defines the normal closure of $a$. 
\end{enum}
In other words, if $b$ is an arbitrary element of $\mathbf{H}$, then $\Psi$ can 
find some nonidentity $a\in b^H$ so that $a^H$ is definable by $\Phi$.  

In their paper, Baker and Wang use this definition to prove another finite 
axiomatisability result. Their result applies to more general algebraic 
structures, but we express it in terms of groups. 

\begin{unthm}[Baker, Wang] Let $\V$ be a variety of groups and suppose that 
$\V$ 
has definable principal normal subgroups. Then, $\V$ is finitely axiomatisable 
if and 
only if $\V_{\si}$ is 
finitely axiomatisable. 
\end{unthm}

A variation on the proof of this theorem yields the following result, whose 
proof we reproduce from McNulty and Wang's unpublished work. Again, the theorem 
holds for more general structures, but we state and prove it in terms of 
groups. 

\begin{thm}\label{mcw}
If $\V$ is a variety of groups and $\V_{\si}$ has definable principal normal 
subgroups, 
then $\V_{\si}$ is finitely axiomatisable relative to $\V$. In particular, if 
$\V$ is finitely axiomatisable, then $\V_{\si}$ is finitely axiomatisable. 
\end{thm}
\begin{proof}
Let $\Sigma$ be a finite set of elementary sentences which axiomatises $\V$, 
and let $\Phi(x,y)$ and $\Psi(x,y)$ be the formulas witnessing that 
$\V_{\si}$ has definable principal normal subgroups. Let $\Theta$ be the 
following set of sentences: $$\Sigma \cup \{\exists u [u \neq 1 \wedge 
\forall z  (z\neq 1 \Rightarrow \exists x (\Phi(u,x) \wedge 
\Psi(x,z)))]\}$$
We claim that $\Theta$ axiomatises $\V_{\si}$. 

On one hand, suppose $\mathbf{S} \in \V_{\si}$. Let $c$ be a 
generator of the monolith of $\mathbf{S}$. So, $c \neq 1$ and $c$ belongs to 
every nontrivial normal subgroup. Now, let $b \in S-\{1\}$. Because $V_{\si}$ 
has definable principal normal subgroups, there exists some nonidentity $a \in 
S$ so 
that $S \models \Psi(a, b)$ and $\Phi(x,a)$ defines the normal closure $a^S$.  
Since $c$ generates the monolith, however, $c \in a^S$ also, and so $S 
\models \Phi(c,a)$. So, $$\mathbf{S} \models \exists u [u \neq 1 
\wedge 
\forall z  (z\neq 1 \Rightarrow \exists x (\Phi(u,x) \wedge 
\Psi(x,z)))]$$ Since $\mathbf{S}$ belongs to $\V$, $\mathbf{S} \models 
\Sigma$ also. Therefore, $\mathbf{S} \models \Theta$. 

Now, suppose $\mathbf{S} \models \Theta$. Then, $\mathbf{S} \in \V$ since 
$\Sigma$ axiomatises $\V$. But also, since $\mathbf{S}$ believes the second 
part of $\Theta$ and since $\Phi$ and $\Psi$ are normal closure formulas, there 
exists $c 
\in S-\{1\}$ so that $c$ is contained within 
any other principal normal subgroup. In particular, the principal normal 
subgroup $c^H$ is contained within any other principal normal subgroup 
of 
$\mathbf{S}$ and so $\mathbf{S}$ is subdirectly irreducible. 
\end{proof}

Thus, in view of Theorem \ref{mcw} and the Oates-Powell theorem, to prove our main result we need only prove the following. 
\begin{thm}\label{dpsc}
Let $\V$ be a variety generated by a finite nilpotent group $\mathbf{G}$. Then, 
$\V_{\si}$ has definable principal normal subgroups. 
\end{thm}
We now introduce some machinery that will let us quantify principal normal 
subgroup inclusion in a first-order way. The set of \emph{conjugate product 
terms} in $x$ of a variety of groups is the 
smallest set $C$ of terms so that 
\begin{itemize} 
\item $1 \in C$
\item If $t \in C$ and $y$ is a variable, then both $(yxy^{-1})t$ and $(yx^{-1}y^{-1})t$ belong to $C$. 
\end{itemize}
The defintion is apt; $C$ is the set of all terms made by taking products of conjugates of $x$ and $x^{-1}$. A sample member of $C$ might be
$$t(x, y_0, y_2, y_7)= y_0 x y_0^{-1} y_2 x^{-1} y_2^{-1} y_7 x y_7^{-1}$$
A \emph{conjugate product polynomial} is a unary polynomial $\pi(x)$ forged from some conjugate product term. We might write $\pi(x,\bar{y})$ if we wish to specify the parameters. So, for instance, in some group $\mathbf{H}$, we might choose members $c_0, c_2, c_7 \in H$ and, from our prior example, obtain the following conjugate product polynomial $$\pi(x)=t(x, c_0, c_2, c_7)=c_0 x c_0^{-1} c_2 x^{-1} c_2^{-1} c_7 x c_7^{-1}$$

Conjugate product polynomials are a powerful tool in groups; they are capable 
of defining principal normal subgroups. The normal closure of an element $a$, 
for instance, is the collection of products of conjugates of $a$ and $a^{-1}$, 
which is precisely the outputs of the sets of conjugate product polynomials in 
$a$. This arms us with a method of defining principal normal subgroups with 
objects that are easily written in first-order logic. 

We refer to a statement of the form $a \in c^H$ as a \emph{membership 
condition}. Membership conditions are our main object of interest in trying to 
establish definable principal normal subgroups, and we now have technology in 
the 
form of conjugate product polynomials to witness them. Our strategy in the 
proof will be to show that these conditions can be witnessed with a limited 
number of variables. This will enable us to quantify the witnessing using a 
first-order statement. In this paper, the \emph{complexity} of a conjugate 
product polynomial refers to the number of conjugates present in the 
product. Our previous example has complexity 3. 
\section{Proving Theorem \ref{dpsc}}

In order to show that $\V_{\si}$ has definable principal normal subgroups, as 
desired, we need two different normal closure formulas. The first, $\Psi(x,y)$, 
to seek out some definable principal normal subgroup of any given principal 
normal subgroup, and the second, 
$\Phi(x,y)$, to do the defining. We will prove the existence of 
$\Phi$ first, using a proof of McNulty and Wang that appears in their 
unpublished paper that they have kindly allowed to be presented here. By an 
\emph{atom} we mean a nontrivial normal subgroup $N$ of $\mathbf{G}$ which does 
not 
properly contain any other nontrivial normal subgroups of $\mathbf{G}$. 

\begin{thm}\label{atoms}
Let $\V$ be the variety generated by a finite group. Then, there is a 
normal closure formula $\Phi(x,y)$ such that for any $\mathbf{H} \in \V$ and 
every 
$c\in H$ such that $c^H$ is an atom in the lattice of normal subgroups of 
$\mathbf{H}$, 
it follows that 
$\Phi(x,c)$ defines $c^H$. 
\end{thm}
\begin{proof}
Let $r$ be a finite upper bound on the size of chief factors in algebras 
belonging to $\V$. Then, we claim that if $c^H$ is an atom 
for some $c \in \mathbf{H}$, then any membership condition of the form 
$a \in c^H$ can be witnessed by a conjugate 
product polynomial of complexity no more than $r$. 

If $a \in c^H$, then $a=g_0 g_1 \cdots g_{n-1}$ where each $g_i$ is some 
conjugate of either 
$c$ or $c^{-1}$. If $n$ is chosen to be as small as 
possible, $$g_0, \quad g_0 g_1, \quad g_0 g_1 g_2,  \ldots, g_0g_1\cdots  
g_{n-1}$$ are 
$n$ distinct elements of $c^H$. Now, since 
$c^H$ is an atom, $c^H/\{1\}$ is a chief factor. 
$|c^H|=|c^H/\{1\}|\leq r$, so $n \leq r$. 

Now, let $T$ be the set of all conjugate product terms in the signature of $\V$ 
whose parameters are chosen from the distinct variables $u_0, \ldots, u_{r-1}$. 
Since there are only finitely many variables being used, $T$ is finite. 
Now, let $\Phi(x,y)$ be the sentence

$$\exists u_0, \ldots, u_{r-1} \left[\bigvee_{t \in T}   t(y, 
\bar{u})\approx 
x\right]$$

$\Phi(x,c)$ now defines $c^H$ whenever $c^H$ is an atom. 
\end{proof}
Theorem \ref{atoms} gives us a normal closure formula that can define any atoms 
in 
any group belonging to $\V$; in particular, for any group in $\V_{\si}$, this 
formula will always define the group's monolith. The second formula 
that we need, $\Psi(x,y)$, will come from the following theorem, which is 
the original work of this paper. 

\begin{thm}\label{descent}
Let $\V$ be a variety generated by a group  
$\mathbf{G}$ of finite exponent $m$ and nilpotence class $k$. Let $\mathbf{S} 
\in \V_{\si}$. Then, 
given 
any $a \in \mathbf{S}$, there is some $b$ 
belonging to the monolith of $\mathbf{S}$ so that the membership condition 
$b \in a^S$ is witnessed by a conjugate 
product polynomial of complexity bounded above in terms of the generating group 
$\mathbf{G}$. 
\end{thm}
\begin{proof}
Since $\mathbf{S} \in \V$, the exponent of $\mathbf{S}$ divides that of 
$\mathbf{G}$, as the equation $x^m=1$ holds throughout $\V$. We also know that 
the nilpotence class $k$ of $\mathbf{S}$ 
is bounded above by that of $\mathbf{G}$; we harmlessly assume it is $k$.  
Denote the upper central series of 
$\mathbf{S}$ as $$\{1\}=Z_0 \triangleleft Z_1 \triangleleft \ldots \triangleleft 
Z_k=\mathbf{S}$$
Note that $Z_1$ is the center of $\mathbf{S}$, which contains the monolith $M$ 
of $\mathbf{S}$. Choose any arbitrary $a\in S$. If $a \in M$, then no more work 
is needed, so we can assume it is not. Label $a=a_k$; now, we will form a 
sequence of elements walking down the steps of the central series that form a 
chain of principal normal subgroups.
Given $a_{i+1} \in Z_{i+1}$, we will seek 
out $a_i$ so that the following hold:
\begin{enum}
\item $a_i \in Z_i$
\item $a_i\neq 1$
\item $a_i  \in a_{i+1}^S$ and this fact is witnessed by a conjugate product 
polynomial of complexity at most $m$. 
\end{enum}
We can certainly find $a_i \in Z_i$ so that $a_i \in a_{i+1}^S$; since 
$\mathbf{S}$ is subdirectly irreducible, any element of the monolith $M$ will 
suffice. We choose $a_i$ from all such possible nonidentity candidates in $Z_i$ 
so that the conjugate product polynomial $\pi_i$ that witnesses $\pi_i(a_{i+1}, 
\bar{c})=a_i$ has minimal possible complexity, and claim that this satisfies 
our above three requirements. The first two are already satisfied, so we need 
only worry about the complexity of $\pi_i$. 

$\pi_i$ takes the form $\pi_i(x, \bar{c})= c_0 x^{\pm 1} c_0^{-1} c_1 x^{\pm 1} c_1^{-1}, \ldots c_n x^{\pm 1} c_n^{-1}$ for some $n$. The structure of this polynomial breaks down into two cases. \\

Case 1) There are both positive and negative conjugates present in $\pi_i$. 	
So, $\pi_i$ contains, somewhere, a product of the form $$c_j x c_j^{-1}c_{j+1} 
x^{-1} c_{j+1}^{-1}$$ (or perhaps the same product with the negative conjugate 
on the left). We claim that these two conjugates are the whole of $\pi_i$, and 
that the element $c_j a_{i+1} c_j^{-1}c_{j+1} a_{i+1}^{-1} c_{j+1}^{-1}$, which 
we will 
temporarily call $a_i^*$, is in fact $a_i$ itself. Indeed, $a_i^*$ cannot be 1; 
if it were, these two conjugates could be removed from $\pi_i$ to preserve the 
given membership condition with a shorter polynomial, contradicting $\pi_i$'s 
minimality. Clearly, $ a_i^* \in a_{i+1}^S$. So all we need to do is show that 
$a_i^* \in Z_i$, and then the minimal complexity of $\pi_i$ will do the rest of 
the work for us. 

Now, $Z_{i+1}/Z_i$ is the center of $\mathbf{S} / Z_i$, so $a_{i+1} /Z_i$ commutes with every member of $\mathbf{S} / Z_i$. So, we have 
\begin{align*}
a_i^*/Z_i&= (c_j a_{i+1} c_j^{-1}c_{j+1} a_{i+1}^{-1} c_{j+1}^{-1})/Z_i\\
&=(c_j/Z_i) (a_{i+1}/Z_i) (c_j^{-1}/Z_i)(c_{j+1}/Z_i)(a_{i+1}^{-1}/Z_i)(c_{j+1}^{-1}/Z_i)\\
&= (a_{i+1}/Z_i)(a_{i+1}^{-1}/Z_i)\\
&= 1/Z_i
\end{align*}
So, $a_i^* \in Z_i$. So, $a_i^*=a_i$, and the complexity of the polynomial 
needed to witness the membership $a_i \in a_{i+1}^S$ is 2, which is certainly 
less than the exponent $m$ of $\mathbf{G}$ unless the variety is trivial. \\

Case 2) The conjugates present in $\pi_i$ are either all positive or all 
negative. We 
assume that the conjugates are all positive; if they are all negative, the 
proof is almost identical. In this case, we claim that the complexity of 
$\pi_i$ is at most $m$. The argument is similar to case 1. Suppose the 
complexity is at least $m$; then, look at $a_i^*=c_0 a_{i+1} c_0^{-1} c_1 
a_{i+1} c_1^{-1} \ldots c_{m-1} a_{i+1} c_{m-1}^{-1}$. We claim that $a_i^*$ 
is, again, $a_i$. As in case 1, $a_i^*$ satisfies criteria 2 and 3, so it only 
remains to show $a_i^* \in Z_i$. Again, we have 
\begin{align*}
a_i^*/Z_i&=(c_0 a_{i+1} c_0^{-1} c_1 a_{i+1} c_1^{-1} \ldots c_{m-1} a_{i+1} c_{m-1}^{-1})/Z_i \\
&=(c_0/Z_i) (a_{i+1}/Z_i) (c_0^{-1}/Z_i) \ldots (c_{m-1}/Z_i) (a_{i+1}/Z_i) (c_{m-1}^{-1}/Z_i) \\
&= (a_{i+1}/Z_i)^m\\
&= 1/Z_i
\end{align*}
since the exponent of any algebra in $\V$ divides $m$. So, again by minimality of $\pi_i$, we have that $a_i^*=a_i$, and so our polynomial has complexity at most $m$.

So, we have a sequence $(a_i)_{i=1}^k$ that walks down through the upper 
central series 
of $\mathbf{S}$, all the way down to $a_1$ which belongs to the center of 
$\mathbf{S}$. We can also walk $a_1$ down to some $a_0$ in the monolith via a 
polynomial $\pi_0$; the same proof suffices, as $Z_1$ is Abelian, so in 
particular its elements commute with every element of $M$. $a_0 \in a^S$, as 
witnessed by the composition of each of the conjugate product polynomials 
$\pi_i$, which is itself a conjugate product polynomial. The complexity of the 
composition is bounded above by $m^k$. This completes the proof. 
\end{proof}

Now, we can complete the proof of Theorem \ref{dpsc}. Let $T$ be the set of all 
conjugate product terms in the signature of $\V$ whose parameters are chosen 
from the distinct variables $u_0, \ldots, u_{m^{k}-1}$. Since the list of 
variables is finite, there are finitely many such terms. Now, let $\Psi(x,y)$ 
be the sentence

$$\exists u_0, \ldots, u_{m^k-1} \left[\bigvee_{t \in T}   t(y, 
\bar{u})\approx x\right]$$

Let $\Phi(x,y)$ be the normal closure formula from Theorem \ref{atoms} that 
defines all atoms of congruence lattices of algebras in $\V$. Together, 
$\Phi(x,y)$ and $\Psi(x,y)$ witness that $\V_{\si}$ has definable 
principal normal subgroups.  \hfill \qedsymbol

\section{Future Research}
A number of natural followup questions to this result present themselves and 
beg to be investigated. For one, the original conjecture of McNulty and Wang is 
still open, as the current paper has only gone partway to solving it. 
\begin{prob} If $\V$ is a variety generated by a finite group $\mathbf{G}$, is 
it true that $\V_{\si}$ is finitely axiomatisable? \end{prob}

If a counterexample can be found to the conjecture, it is natural to want to 
know how far the finite axiomatisability can be taken. Is nilpotence the best 
we can do, or are there broader or perhaps unrelated classes of groups for 
which our result hold?

\begin{prob} If $\V$ is a variety generated by some group $\mathbf{G}$, what 
	conditions have to be met by $\mathbf{G}$ in order for $\V_{\si}$ to be 
	finitely axiomatisable? \end{prob}

Universal algebraists will of course wish to extend this result to more general 
algebras besides just groups. We define an 
\emph{algebra} to be a nonempty set $A$ along with finitary operations which 
define functions on $A$. Groups, rings, vector spaces and lattices all satsify 
this rather broad definition. Departing the isle of groups in favour of such 
general waters robs us of the concept of normal subgroups, so in universal 
algebra the attention is usually given to congruence relations. 

A \emph{congruence relation} on an algebra $\mathbf{A}$ is  
the 
relational kernel of some homomorphism. That is, given a homomorphism $h$, the 
associated congruence is the relational kernel $\{\langle a, b \rangle : 
h(a)=h(b)\}$. 
Congruences are also precisely the 
equivalence relations on $\mathbf{A}$ that are also subalgebras of 
$\mathbf{A}^2$. In groups, the equivalence classes of congruences are the left
cosets of the normal subgroup that is the kernel of the homomorphism in 
question. The congruence class containing 1 
is the kernel of the homomorphism. In algebras in general, we lack a unit 
element, and so the congruence relation as a kernel must replace the normal 
subgroup in 
most 
discussion.  

The concept of the commutator in group theory has also
been extended to a general commutator that exhibits similar properties in 
general algebras, at least in the case when the congruence lattices are 
modular, as is the case in groups. We refer readers to Freese and McKenzie's 
1987 work 
\cite{FreeseMcKenzie1987} for an 
exhaustive discussion of commutator theory. This commutator enables an echo of 
concepts like Abelianness, nilpotence and 
solvability in varieties of algebras whose congruence lattices are modular. The 
algebraic commutator provides us with a lower central 
series. The concept of Abelian algebras similarly provides us with an upper 
central series, and as in groups, these two series have the same length. 

Unfortunately, of the three characterisations of nilpotence we described 
earlier, the upper and lower central series are the only ones that extend to 
the algebraic commutator. The characterisation of a nilpotent group as one 
which is the product of its Sylow subgroups does not extend to general 
algebras, even in the otherwise well-behaved congruence modular varieties. 
What's worse, the notion of conjugate product polynomials is 
not so easily translated to algebras either, so the method of this paper does 
not seem to 
generalise beyond group theory. However, other methods might bear some fruit. 
The author has proved 
that the result extends to varieties of nilpotent algebras with the added 
hypothesis that the generating algebra is a product of algebras of prime power 
order.  This 
proof should turn up in future publication. Nilpotent algebras in general 
remain a tougher nut to crack. 

\begin{prob}If $\V$ is a congruence modular variety generated by a finite 
nilpotent 
algebra 
$\mathbf{A}$, is it true that $\V_{\si}$ is finitely axiomatisable? \end{prob}

\bibliographystyle{srtnumbered}
\bibliography{NPref}

\begin{thebibliography}{1}
\newcommand{\enquote}[1]{`#1'}

\bibitem{BakerWang2002}
Kirby~A. Baker and Ju~Wang.
\newblock \enquote{Definable principal subcongruences}.
\newblock {\em Algebra Universalis\/} {\bf 47}~(2) (2002), 145--151.

\bibitem{Birkhoff1935}
Garrett Birkhoff.
\newblock \enquote{On the structure of abstract algebras}.
\newblock {\em Mathematical Proceedings of the Cambridge Philosophical
  Society\/} {\bf 31}~(4) (1935), 433–454.

\bibitem{Birkhoff1944}
Garrett Birkhoff.
\newblock \enquote{Subdirect unions in universal algebra}.
\newblock {\em Bull. Amer. Math. Soc.\/} {\bf 50} (1944), 764--768.

\bibitem{DummitFoote}
David~Steven Dummit and Richard~M. Foote.
\newblock {\em Abstract Algebra\/} (John Wiley and Sons, Inc., Hoboken, NJ,
  2004), 3rd edn.

\bibitem{FreeseMcKenzie1987}
Ralph Freese and Ralph McKenzie.
\newblock {\em Commutator theory for congruence modular varieties\/}, {\em
  London Mathematical Society Lecture Note Series\/}, Volume 125 (Cambridge
  University Press, Cambridge, 1987).

\bibitem{Lyndon1952}
R.~C. Lyndon.
\newblock \enquote{Two notes on nilpotent groups}.
\newblock {\em Proc. Amer. Math. Soc.\/} {\bf 3} (1952), 579--583.

\bibitem{Neumann}
H.~Neumann.
\newblock {\em Varieties of Groups\/}.
\newblock Ergebnisse der Mathematik und ihrer Grenzgebiete. 2. Folge (Springer
  Berlin Heidelberg, 1967).

\bibitem{OatesPowell1964}
Sheila Oates and M.~B. Powell.
\newblock \enquote{Identical relations in finite groups}.
\newblock {\em J. Algebra\/} {\bf 1} (1964), 11--39.

\end{thebibliography}

\end{document}